\documentclass[11pt]{article}

\usepackage[T1]{fontenc}
\usepackage{hyperref, amsfonts, amsmath, amssymb, amsthm, fullpage, bm, mathtools, microtype, mleftright}
\usepackage[noabbrev, capitalize]{cleveref}

\newtheorem{theorem}{Theorem}
\newtheorem{lemma}[theorem]{Lemma}
\newtheorem{proposition}[theorem]{Proposition}

\AddToHook{env/lemma/begin}{\crefalias{theorem}{lemma}}
\AddToHook{env/proposition/begin}{\crefalias{theorem}{proposition}}

\newtheorem*{claim*}{Claim}

\newenvironment{claimproof}[1][Proof]
  {\begin{proof}[#1]}
  {\end{proof}}

\DeclareMathOperator{\tr}{tr}
\DeclareMathOperator{\rank}{rank}

\DeclareMathOperator{\range}{im}
\DeclareMathOperator{\proj}{Proj}
\DeclareMathOperator{\sgn}{sgn}
\DeclarePairedDelimiter\abs{\lvert}{\rvert}
\DeclarePairedDelimiter\norm{\lVert}{\rVert}
\DeclarePairedDelimiter\ip{\langle}{\rangle}
\DeclarePairedDelimiter\floor{\lfloor}{\rfloor}
\DeclarePairedDelimiter\ceil{\lceil}{\rceil}

\newcommand{\dset}[2]{\left\{{#1}\colon{#2}\right\}}
\newcommand{\sset}[1]{\left\{{#1}\right\}}

\newcommand{\eps}{\varepsilon}
\newcommand{\qf}[2]{#2^\T #1 #2}

\newcommand{\R}{\mathbb{R}}
\newcommand{\E}{\mathbb{E}}
\newcommand{\T}{\intercal}

\newcommand{\cube}{[-1,1]^n}
\newcommand{\csec}{[-1,1]^n \cap K}
\newcommand{\qube}{\sset{\pm 1}^n}
\newcommand{\zube}{\sset{0,\pm 1}^n}
\newcommand{\qsec}{\zube \cap K}

\newcommand{\quadand}{\quad\text{and}\quad}
\newcommand{\quadas}{\quad\text{almost surely}}

\title{Unbalancing unit vectors}
\author{
  Zilin Jiang\thanks{School of Mathematical and Statistical Sciences, and School of Computing and Augmented Intelligence, Arizona State University, Tempe, AZ 85281, USA. Email: \texttt{zilinj@asu.edu}. Supported in part by the Simons Foundation through its Travel Support for Mathematicians program and by U.S.\ taxpayers through NSF grant 2451581.}
  \and Jeck Lim\thanks{Hausdorff Center for Mathematics, University of Bonn, 53115 Bonn, Germany. Email: \texttt{jlim1@uni-bonn.de}.}
  \and Skand Parvatikar\thanks{Department of Mathematics, Statistics, and Computer Science, University of Illinois Chicago, Chicago, IL 60607, USA. Email: \texttt{sparvati@asu.edu}.}
}
\date{}

\begin{document}

\maketitle

\begin{abstract}
  We show that for every $n$ unit vectors $v_1, \dots, v_n$ in the $d$-dimensional Euclidean space, there exist signs $\varepsilon_1, \dots, \varepsilon_n \in \{\pm 1\}$ such that $\lVert \varepsilon_1 v_1 + \dots + \varepsilon_n v_n \rVert \ge \sqrt{2n - d}$, and we characterize the equality cases.
\end{abstract}

\section{Introduction} \label{sec:intro}

In 1963, Dvoretzky \cite{D63} asked how well one can balance $n$ unit vectors in a normed vector space $E$: what is the smallest constant $c = c(E,n)$ such that every choice of unit vectors $v_1, \dots, v_n \in E$ admits signs $\eps_1, \dots, \eps_n \in \sset{\pm1}$ with
\[
  \norm{\eps_1v_1 + \dots + \eps_n v_n} \le c?
\]
For $\R^d$ equipped with the Euclidean norm, Spencer \cite{S81} proved that $c(\R^d, n) \le \sqrt{d}$, which is sharp when $n \ge d$ and $n - d$ is even. For the related problem of balancing vectors of norm at most $1$ in general finite-dimensional normed spaces, see B\'ar\'any and Grinberg \cite{BG81}.

A related question is the Koml\'os conjecture, which asks whether vectors of Euclidean norm at most $1$ always admit a signed sum whose $\ell_\infty$-norm is bounded by a universal constant. Guo, Fang, and Lu \cite{GFL26} recently proved this conjecture, and Karingula and Lovett \cite{KL26} subsequently gave an elementary proof.

In this paper, we focus on unbalancing unit vectors in Euclidean spaces --- given $n$ and $d$, determine the largest constant $C = C(\R^d, n)$ such that for every $n$ unit vectors $v_1, \dots, v_n$ in $\R^d$, there exist signs $\eps_1, \dots, \eps_n \in \sset{\pm1}$ such that
\[
  \norm{\eps_1v_1 + \dots + \eps_n v_n} \ge C.
\]
The same definition gives $C(E,n)$ for any $d$-dimensional normed space $E$. Ambrus and Grundbacher \cite{AG25} proved that $C(E,n)\ge\ceil{n/d}$, with equality when $E=\R^d$ is equipped with the infinity norm.

A simple probabilistic argument gives $C(\R^d,n) \ge \sqrt{n}$: the expected squared norm of  a sum with independent random signs is $n$. An orthonormal set attains this bound when $n \le d$. For $n\ge d$, the known bounds \cite[Proposition 2]{AG25} give
\[
  \sqrt{\frac{2}{\pi}}\,\frac{n}{\sqrt d}
  \le C(\R^d,n) \le \frac{n}{\sqrt d}.
\]
As soon as $n=d+1$, the unit vectors become linearly dependent. Ambrus and Nietert \cite{AN19} conjectured that, for every $d\ge2$,
\[
  C(\R^d,d+1)=\sqrt{d+2},
\]
with equality precisely for the union of unit vectors forming a centered regular simplex in a nonzero even-dimensional subspace $V$ and an orthonormal basis of $V^\perp$, up to \emph{switching}, that is, replacing some vectors $v_i$ by $-v_i$.\footnote{Ambrus and Nietert credit Alexandr Polyanskii with correcting an earlier version of the conjecture in \cite{BFGK18}.} This configuration attains the proposed bound, so the conjecture asserts both its optimality and the characterization of all extremizers.

Ambrus and Gonz\'alez Merino \cite{AG21} reiterated the conjecture and proved the lower bound when $d$ is even and the vectors sum to zero. Fu, Wang, and Yan \cite[Proposition 2.6]{FWY23} settled the case $d=3$, including the equality characterization.

We resolve both parts of the conjecture and extend the lower bound and equality characterization to general $n\ge d$. We begin with the lower bound.

\begin{theorem} \label{thm:main}
  For every $n$ unit vectors $v_1, \dots, v_n \in \R^d$ with $n \ge d$, there exist signs $\eps_1, \dots, \eps_n \in \sset{\pm1}$ such that
  \[
    \norm*{\eps_1 v_1 + \dots + \eps_n v_n} \ge \sqrt{2n-d}.
  \]
\end{theorem}

When $n=d+1$, this gives the conjectured lower bound $\sqrt{d+2}$. To see when the bound is sharp more generally, we combine regular simplices in mutually orthogonal even-dimensional subspaces with an orthonormal set.

\begin{proposition} \label{prop:equalities}
  Let $n \ge d$ be positive integers, and let
  \[
    \R^d=V_0\oplus V_1\oplus\dots\oplus V_{n-d}
  \]
  be an orthogonal decomposition such that $V_1,\dots,V_{n-d}$ have positive even dimensions, with $V_0$ possibly zero dimensional. If $v_1,\dots,v_n$ consist of an orthonormal basis of $V_0$ and the vertices of a centered regular simplex inscribed in the unit sphere of $V_i$ for each $i \in \sset{1, \dots, n-d}$, then
  \[
    \norm*{\eps_1 v_1 + \dots + \eps_n v_n} \le \sqrt{2n-d}, \text{ for every } \eps_1, \dots, \eps_n \in \sset{\pm1}.
  \]
\end{proposition}

Taking $n-d$ mutually orthogonal equilateral triangles and an orthonormal set of $3d-2n$ additional vectors shows that for every $d$, and every $n \in \sset{d, \dots, \floor{3d/2}}$,
\[
  C(\R^d, n) = \sqrt{2n-d}.
\] Our final result shows that the configurations in \cref{prop:equalities} exhaust the equality cases, up to switching.

\begin{theorem} \label{thm:equalities}
  Let $v_1, \dots, v_n \in \R^d$ be unit vectors with $n \ge d$. If
  \[
    \norm*{\eps_1 v_1 + \dots + \eps_n v_n} \le \sqrt{2n-d}, \text{ for every } \eps_1, \dots, \eps_n \in \sset{\pm1},
  \]
  then, after switching, the vectors form a configuration described in \cref{prop:equalities}.
\end{theorem}

The rest of the paper is organized as follows. In \cref{sec:proof}, we prove \cref{thm:main} using Ball and Prodromou's combinatorial version of Vaaler's cube slicing theorem \cite{BP09}. In \cref{sec:equalities}, we verify \cref{prop:equalities} and prove the characterization in \cref{thm:equalities}. We conclude with further remarks in \cref{sec:discussion}.

\section{The lower bound} \label{sec:proof}

To prove \cref{thm:main}, let $G$ be the Gram matrix of the unit vectors, let $I_n$ be the identity matrix of order $n$, and set $M=G-I_n$. The diagonal of $M$ is zero, that is $M$ is \emph{hollow}, and the multiplicity of $-1$ as an eigenvalue of $M$ is at least $n-d$. The strategy, roughly speaking, is to find a point $v \in \cube$ with a large value of the quadratic form $\qf{M}{v}$, and round it to the vertices in $\qube$ while maintaining the large value.

The first ingredient is a beautiful result of Ball and Prodromou \cite{BP09}: every cube section supports a random point whose second moment is at least the identity on its underlying subspace. This is a combinatorial analogue of Vaaler's cube slicing theorem \cite{V79}, which states that every $k$-dimensional central section of $[-1,1]^n$ has $k$-dimensional volume at least $2^k$. We write $\proj_K$ for the orthogonal projection onto a subspace $K$.

\begin{theorem}[Ball and Prodromou {\cite[Theorem 1]{BP09}}] \label{thm:bp1}
  For every subspace $K$ of $\R^n$, there exists a random vector $V$ supported on $\csec$ such that $\E[VV^\T] \succeq \proj_K$. \qed
\end{theorem}

The second ingredient rounds a point of $\cube$ to a vertex of $\qube$ without decreasing the value of the quadratic form.

\begin{lemma} \label{lem:rounding}
  For every hollow matrix $M$ of order $n$, and every vector $v \in \cube$, there exists a vector $\eps \in \qube$ such that $\qf{M}{\eps} \ge \qf{M}{v}$.
\end{lemma}

\begin{proof}
  Let $V$ be a random vector supported on $\qube$ whose coordinates $V_1, \dots, V_n$ are independent $\pm 1$ random variables such that $\E[V_i] = v_i$ for every $i$. Since $M$ is hollow, we have
  \[
    \E[\qf{M}{V}] = \qf{M}{v},
  \]
  which warrants a vector $\eps \in \qube$ that satisfies $\qf{M}{\eps} \ge \qf{M}{v}$.
\end{proof}

Together, these two results yield the following matrix lemma, which immediately implies \cref{thm:main}.

\begin{lemma} \label{lem:main}
  For every symmetric hollow matrix $M$, there exists a vector $\eps \in \qube$ such that $\qf{M}{\eps}$ is at least the sum of positive eigenvalues of $M$, which, in particular, is at least the multiplicity of $-1$ as an eigenvalue of $M$.
\end{lemma}

\begin{proof}
  Let $M = Q + R$ be the decomposition of $M$ into its positive and negative parts, that is,
  \[
    Q = \sum_{i \colon \lambda_i > 0}\lambda_i u_i u_i^\T\quad \text{and} \quad R = \sum_{i \colon \lambda_i < 0}\lambda_i u_i u_i^\T,
  \]
  where $u_1, \dots, u_n$ form an eigenbasis of $M$, and $\lambda_1, \dots, \lambda_n$ are the corresponding eigenvalues. In particular, $Q$ is positive semidefinite, $\range Q \subseteq \ker R$, and $\tr Q$ is the sum of the positive eigenvalues of $M$.

  Set $K = \range Q$, and let $V$ be the random vector supported on $\csec$ given by \cref{thm:bp1} for $K$. Since $Q$ is positive semidefinite and $\E[VV^\T]\succeq\proj_K$, taking the trace against $Q$ gives\footnote{For positive semidefinite matrices $A$ and $B$, $\tr(AB)=\tr\mleft(A^{1/2}BA^{1/2}\mright)\ge0$.}
  \[
    \E[\qf{Q}{V}] = \tr(Q\E[VV^\T]) \ge \tr(Q\proj_K) = \tr Q.
  \]
  Thus some $v \in \csec$ satisfies $v^\T Qv\ge\tr Q$. Since $v \in K \subseteq \ker R$, we have
  \[
    \qf{M}{v} = \qf{Q}{v} \ge \tr Q.
  \]
  According to \cref{lem:rounding}, there exists $\eps \in \qube$ such that $\qf{M}{\eps} \ge \tr Q$. Finally, if $m$ is the multiplicity of $-1$ as an eigenvalue of $M$, then $\tr M=0$ because $M$ is hollow, and hence $\tr Q=-\tr R\ge m$.
\end{proof}

\begin{proof}[Proof of \cref{thm:main}]
  Let $G$ be the Gram matrix of $v_1,\dots,v_n$. The matrix $M=G-I_n$ is hollow, and the multiplicity of $-1$ as an eigenvalue of $M$ is $n-\rank G$. By \cref{lem:main}, there exists $\eps \in \qube$ such that $\qf{M}{\eps}$ is at least the sum of the positive eigenvalues of $M$, hence at least $n-\rank G\ge n-d$. Therefore,
  \[
    \norm*{\eps_1 v_1 + \dots + \eps_n v_n}^2 = \qf{G}{\eps} = n + \qf{M}{\eps} \ge 2n-d. \qedhere
  \]
\end{proof}

\section{Equality cases} \label{sec:equalities}

We start with the straightforward verification of the equality cases in \cref{prop:equalities}.

\begin{proof}[Proof of \cref{prop:equalities}]
  Every signed sum of an orthonormal basis of $V_0$ has squared norm $\dim V_0$. By orthogonality of the subspaces, it suffices to show that every signed sum of the vertices $v_0,\dots,v_k$ of a centered regular simplex on the unit sphere of an even-dimensional subspace $V_i$ has squared norm at most $k+2$, where $k$ is the dimension of $V_i$. Indeed, the squared norm of any signed sum in \cref{prop:equalities} would therefore be at most
  \[
    \dim V_0+\sum_{i=1}^{n-d}(\dim V_i+2)=d+2(n-d)=2n-d.
  \]

  Note that distinct vertices satisfy $\ip{v_i,v_j}=-1/k$. Since $k+1$ is odd, any signs $\eps_0,\dots,\eps_k \in \sset{\pm 1}$ satisfy $\abs*{\eps_0 + \dots + \eps_k}\ge1$. Hence
  \[
    \norm*{\eps_0v_0 + \dots + \eps_kv_k}^2 = k+1-\frac{1}{k}\sum_{i \ne j}\eps_i\eps_j = \frac{1}{k}\left((k+1)^2-\left(\sum_i \eps_i\right)^2\right) \le k+2. \qedhere
  \]
\end{proof}

We now show that these are the only equality cases up to switching. In the proof of \cref{thm:main}, we decompose $M=G-I_n$ into positive and negative parts $Q$ and $R$, and find $x\in\cube\cap K$, where $K=\range Q$, such that $\qf{Q}{x}\ge\tr Q$. The next lemma describes what happens when the reverse inequality holds throughout $\cube\cap K$.

\begin{lemma} \label{lem:extremal}
  Let $Q$ be a positive semidefinite matrix of order $n$, and let $K=\range Q$. If
  \[
    \qf{Q}{x} \le \tr Q, \text{ for every } x \in \csec,
  \]
  then there exists a random vector $V$ supported on $\qsec$ such that
  \[
    \E\left[VV^\T\right] = \proj_K \quadand
    \qf{Q}{V} = \tr Q \quadas.
  \]
\end{lemma}

\begin{proof}
  According to \cref{thm:bp1}, let $V$ be a random vector on $\csec$ such that $\E[VV^\T] \succeq \proj_K$. Since $V^\T QV\le\tr Q$ almost surely, we have
  \[
    \tr Q
    \le \tr(Q\E[VV^\T])
    = \E[V^\T QV]
    \le \tr Q.
  \]
  Since $V$ is supported on $K$, the operator $\E[VV^\T] - \proj_K$ acts on $K$. It is positive semidefinite, and $Q$ is positive definite on $K$, so equality forces $\E[VV^\T] = \proj_K$.\footnote{If $A$ is positive definite and $B$ is positive semidefinite, then $\tr(AB)=0$ implies that $B$ is zero: the positive semidefinite matrix $A^{1/2}BA^{1/2}$ has trace zero, hence vanishes, and $A^{1/2}$ is invertible.} Furthermore, $\qf{Q}{V} = \tr Q$ almost surely.

  It remains to show that $V$ is supported on $\zube$. Note that every support point of $V$ achieves the maximum $\tr Q$ of $x\mapsto\qf{Q}{x}$ over $\csec$. Since this quadratic form is strictly convex on $K$, each such maximizer is a vertex of $\csec$. Write $d=\dim K$. A vertex must have at least $d$ coordinates equal to $\pm1$: otherwise, a nonzero direction in $K$ vanishing on those coordinates would allow a perturbation in both directions within $\csec$. However,
  \[
    \E\norm{V}^2  = \E[\tr\mleft(VV^\T\mright)] = \tr \E[VV^\T] = \tr \proj_K = d.
  \]
  Hence $V$ has exactly $d$ nonzero coordinates, and they are $\pm1$, almost surely.
\end{proof}

In the equality case, \cref{lem:extremal} will provide maximizers of $x\mapsto\qf{M}{x}$ over $\cube$ that lie in $\zube$. The next lemma derives constraints on the entries of $M$ from any such maximizer.

\begin{lemma} \label{lem:maximizer}
  For every symmetric hollow matrix $M$ of order $n$, and every maximizer $v$ of $x \mapsto \qf{M}{x}$ over $\cube$, if $v \in \zube$, then
  \[
    v_i(Mv)_i \ge \sum_{j \colon v_j = 0}\abs{M_{ij}}, \text{ for every }i,
  \]
  and moreover $(Mv)_j = 0$ for every $j$ with $v_j = 0$.
\end{lemma}

\begin{proof}
  If $v_j=0$, then $v\pm e_j\in\cube$, so maximality and $M_{jj}=0$ give 
  \[
    0\ge\qf{M}{(v\pm e_j)}-\qf{M}{v}=\pm2(Mv)_j,
  \]
  which implies $(Mv)_j=0$.

  Let $V$ agree with $v$ at every nonzero coordinate of $v$, and choose its remaining coordinates independently and uniformly at random from $\sset{\pm1}$. Since $M$ is hollow, the expectation satisfies $\E[\qf{M}{V}]=\qf{M}{v}$. Thus every support point of $V$ achieves the maximum $\qf{M}{v}$ of $x\mapsto\qf{M}{x}$ over $\cube$. Fix an arbitrary $i \in \sset{1,\dots, n}$. We break the rest of the proof into two cases.

  \medskip\noindent\textbf{Case 1: $v_i=0$.} Fix $j$ with $v_j=0$. If $j=i$, then $M_{ij}=0$ because $M$ is hollow. Otherwise, choose a support point $u$ of $V$ with $u_i=u_j=-1$. The four vectors $u$, $u+2e_i$, $u+2e_j$, and $u+2e_i+2e_j$ are all support points and hence attain this maximum. Therefore,
  \[
    0=\qf{M}{(u+2e_i+2e_j)}-\qf{M}{(u+2e_i)}-\qf{M}{(u+2e_j)}+\qf{M}{u}=8M_{ij}.
  \]
  Thus $M_{ij}=0$ whenever $v_j=0$, and the desired inequality holds with equality.

  \medskip\noindent\textbf{Case 2: $v_i \ne 0$.} Choose a support point $w$ of $V$ so that $v_iM_{ij}w_j=-\abs{M_{ij}}$ whenever $v_j=0$. Since $w$ attains this maximum, $w_i=v_i$, and $M_{ii}=0$, we have
  \[
      0 \le \qf{M}{w}-\qf{M}{(w-2v_ie_i)}=4v_i(Mw)_i =4\left(v_i(Mv)_i-\sum_{j\colon v_j=0}\abs{M_{ij}}\right). \qedhere
  \]
\end{proof}

Here and throughout, $\abs{M}$ denotes the matrix obtained by taking the absolute value of each entry of $M$. We use the same notation for vectors.

In the proof of \cref{thm:equalities}, the preceding lemmas will help produce a nonnegative vector $q$ such that $\abs{M}q\le q$. This vector is formed from the diagonal entries of $Q$, the positive part of $M$, and is positive wherever $M$ has a nonzero row. The next lemma shows that these conditions allow us to switch $M$ to an entrywise nonpositive matrix, so that the Perron--Frobenius theorem applies.

\begin{lemma} \label{lem:switching}
  Let $M$ be a symmetric hollow matrix whose negative eigenvalues are all equal to $-1$. Suppose there exists a nonnegative vector $q$ such that $\abs{M}q\le q$ and $q_i>0$ whenever the $i$-th row of $M$ is nonzero. Then there exists a diagonal matrix $D$ with $\pm1$ diagonal entries such that $DMD=-\abs{M}$.
\end{lemma}

\begin{proof}
  Using the connected components of the support graph of $M$, we can partition $M$ into a block diagonal matrix, where each diagonal block is either a zero matrix or a nonzero irreducible matrix. By focusing on a diagonal block, we may assume that $M$ itself is nonzero and irreducible, and hence $q$ is positive. Let $y$ be a positive Perron vector of $\abs{M}$, so that $\abs{M}y=\rho(\abs{M})y$, where $\rho(\abs{M})$ is the spectral radius of $\abs{M}$. Note that
  \[
    \rho(\abs{M})q^\T y = q^\T \abs{M} y \le q^\T y.
  \]
  Since $q$ is positive, we have $q^\T y > 0$, which implies $\rho(\abs{M}) \le 1$.

  Since $M$ is nonzero and hollow, it has a negative eigenvalue, which must be $-1$. Let $x$ be a corresponding eigenvector. The inequality $\abs{M}\abs{x}\ge\abs{Mx}=\abs{x}$ and the positivity of $y$ give
  \[
    y^\T\abs{x}\le y^\T\abs{M}\abs{x}=\rho(\abs{M})y^\T\abs{x}\le y^\T\abs{x}.
  \]
  Since $y^\T\abs{x}>0$, we have $\rho(\abs{M})=1$. As $y$ is positive, equality also forces $\abs{M}\abs{x}=\abs{x}$. Irreducibility then gives that $\abs{x}$ is positive. Equality in the coordinatewise triangle inequality $\abs{Mx}\le\abs{M}\abs{x}$ implies that $\sgn(M_{ij}x_j)=-\sgn(x_i)$ whenever $M_{ij}\ne0$. Hence the diagonal matrix $D$ with entries $D_{ii}=\sgn(x_i)$ satisfies $DMD=-\abs{M}$.
\end{proof}

To finish the characterization, we apply \cref{lem:switching} and consider the nontrivial connected components of the support graph of $M$ separately. Each corresponding block is then irreducible, hollow, and entrywise nonpositive. The next lemma determines its form under the equality conditions.

\begin{lemma} \label{lem:rigidity}
  For every symmetric irreducible hollow matrix $M$ of order $n$ with nonpositive entries, if $-1$ is the unique negative eigenvalue of $M$, and $M$ satisfies
  \[
    \eps^\T M \eps \le 1, \text{ for every }\eps \in \sset{\pm 1}^n,
  \]
  then the order $n$ is odd, and
  \[
    M = \frac{I_n - J_n}{n-1}.
  \]
\end{lemma}

\begin{proof}
  By the Perron--Frobenius theorem, $1$ is the Perron--Frobenius eigenvalue of $-M$. Let $u$ be the corresponding unit eigenvector with positive coordinates. Set $Q=M+uu^\T$ and $K=\range Q$. Since $-1$ is the unique negative eigenvalue of $M$, we know that $Q$ is positive semidefinite and $u\in\ker Q$. Moreover, $\tr Q=1$ because $M$ is hollow.

  Applying \cref{lem:rounding} to $M$, we know that $\qf{M}{x}\le1$ for every $x\in\cube$. Since $u\in\ker Q$, this implies that $\qf{Q}{x}=\qf{M}{x}\le1=\tr Q$ for every $x\in\csec$. Then \cref{lem:extremal} yields a random vector $V$ supported on $\qsec$ such that
  \[
    \E[VV^\T]=\proj_K \quadand \qf{Q}{V}=\tr Q=1\quadas.
  \]
  Since $V$ is supported on $K$, we have $\qf{M}{V}=1$ almost surely. Hence every support point of $V$ achieves the maximum $1$ of the function $x\mapsto\qf{M}{x}$ over $\cube$.

  \begin{claim*}
    We have $\abs{M}\bm{1}=\bm{1}$ and $K=\bm{1}^\perp$. Almost surely, $V$ has a unique zero coordinate $J$, which is uniformly distributed on $\sset{1,\dots,n}$, and
    \[
      (MV)_i=V_i\abs{M_{iJ}} \text{ for every }i.
    \]
  \end{claim*}
  \begin{claimproof}[Proof of Claim]
    By \cref{lem:maximizer}, for every $i$,
    \begin{equation} \label{eq:mij-ineq}
      V_i(MV)_i\ge\sum_{j\colon V_j=0}\abs{M_{ij}} \quadas.
    \end{equation}
    Set $s_j=\sum_i\abs{M_{ij}}$. Summing over $i$ gives
    \begin{equation} \label{eq:s-ineq}
      1=\qf{M}{V}\ge\sum_{j\colon V_j=0}s_j \quadas.
    \end{equation}
    Since $K\subseteq u^\perp$, we have $\proj_K+uu^\T\preceq I_n$. Thus $\E[VV^\T]=\proj_K$ gives for every $j$ that
    \begin{equation} \label{eq:pos-p}
      \Pr(V_j=0)=1-(\proj_K)_{jj}\ge u_j^2>0.
    \end{equation}
    Consequently $s_j\le1$ for every $j$. But $\abs{M}u=u$, so
    \[
      \sum_j u_js_j=u^\T\abs{M}\bm{1}=u^\T\bm{1}=\sum_j u_j.
    \]
    Since every $u_j$ is positive, all $s_j=1$. Symmetry gives $\abs{M}\bm{1}=\bm{1}$, and irreducibility gives $u=\bm{1}/\sqrt n$.

    Since $s_j=1$ for every $j$, \eqref{eq:s-ineq} shows that $V$ has at most one zero coordinate almost surely. On the other hand, using \eqref{eq:pos-p}, we obtain
    \[
      \E[\abs{\dset{j}{V_j=0}}] = \sum_j \Pr(V_j = 0) = n - \tr \proj_K \ge \sum_j u_j^2 = 1.
    \]
    Thus $V$ has exactly one zero coordinate $J$ almost surely, and $n-\dim K = n - \tr \proj_K = 1$. Since $K\subseteq u^\perp$, we have $K=u^\perp=\bm{1}^\perp$. Moreover,
    \[
      \Pr(J=j)=1-(\proj_K)_{jj}=u_j^2=\frac1n.
    \]
    Since $s_J=1$ and $\qf{M}{V}=1$ almost surely, equality holds in \eqref{eq:s-ineq}, hence equality holds in \eqref{eq:mij-ineq} for every $i$. For $i\ne J$, since $V_i\in\sset{\pm1}$, this gives $(MV)_i=V_i\abs{M_{iJ}}$. For $i=J$, \cref{lem:maximizer} gives $(MV)_J=0=V_J\abs{M_{JJ}}$.
  \end{claimproof}

  Since $V \in K = \bm{1}^\perp$ and its $n-1$ nonzero coordinates are $\pm 1$, we conclude that $n$ is odd. The claim and $M_{JJ}=0$ give
  \[
    \E\norm{MV}^2=\E\sum_i M_{iJ}^2=\frac{1}{n}\tr(M^2).
  \]
  On the other hand, $\E[VV^\T]=\proj_K=I_n-uu^\T$ and $Mu=-u$, so
  \[
    \E\norm{MV}^2=\tr(M^2\proj_K)=\tr(M^2)-1.
  \]
  Thus $\sum_{i\ne j}M_{ij}^2=\tr(M^2)=n/(n-1)$. Since $\abs{M}\bm{1}=\bm{1}$, we also have $\sum_{i\ne j}\abs{M_{ij}}=n$. Equality holds in the Cauchy--Schwarz inequality over the off-diagonal entries of $\abs{M}$, forcing $\abs{M_{ij}}=1/(n-1)$ for every $i\ne j$. Therefore $M=(I_n-J_n)/(n-1)$.
\end{proof}

We now use \cref{lem:switching,lem:rigidity} to characterize the Gram matrix of an equality case.

\begin{proof}[Proof of \cref{thm:equalities}]
  Let $G$ be the Gram matrix of $v_1,\dots,v_n$. If $\rank G<d$, we apply \cref{thm:main} in the span of these vectors to obtain a signed sum with squared norm at least $2n-\rank G>2n-d$, a contradiction.

  Hereafter we assume that $\rank G=d$. Set $m=n-d$, let $M=G-I_n$, and decompose $M=Q+R$ as in the proof of \cref{lem:main}. In particular, $Q$ is positive semidefinite, $\tr Q$ is the sum of the positive eigenvalues of $M$, $R$ is negative semidefinite, the multiplicity of $-1$ as an eigenvalue of $R$ is $m$, and $\range Q \subseteq \ker R$. The hypothesis gives $\qf{M}{\eps}\le m$ for every $\eps\in\qube$. By \cref{lem:main}, some vector $\eps \in \qube$ satisfies $\qf{M}{\eps}\ge\tr Q\ge m$. Thus $\tr Q=m$, and $R$ has exactly $m$ nonzero eigenvalues, all equal to $-1$.

  Set $K = \range Q$. Applying \cref{lem:rounding} to $M$, we know that $\qf{M}{x}\le m$ for every $x\in\cube$. Since $K \subseteq\ker R$, this implies that $\qf{Q}{x}=\qf{M}{x}\le m=\tr Q$ for every $x\in\csec$. Then \cref{lem:extremal} yields a random vector $V$ supported on $\qsec$ such that
  \[
    \E[VV^\T]=\proj_K \quadand \qf{Q}{V}=\tr Q=m\quadas.
  \]

  Since $V$ is supported on $K$, we have $\qf{M}{V}=m$ almost surely. Hence every support point of $V$ achieves the maximum $m$ of $x\mapsto\qf{M}{x}$ over $\cube$.

  \begin{claim*}
    Let $q\in\R^n$ be defined by $q_i=Q_{ii}$. Then $q$ is nonnegative, $q\ge\abs{M}q$, and $q_i>0$ whenever the $i$-th row of $M$ is nonzero.
  \end{claim*}
  \begin{claimproof}[Proof of Claim]
    By \cref{lem:maximizer}, for every $i$,
    \[
      V_i(MV)_i\ge\sum_{j\colon V_j=0}\abs{M_{ij}}\quadas.
    \]
    Taking expectations gives
    \begin{equation} \label{eq:exp-ineq}
      \sum_j M_{ij}\E[V_iV_j]\ge\sum_j\abs{M_{ij}}\Pr(V_j=0).
    \end{equation}
    Since $\E[VV^\T]=\proj_K$ and $M\proj_K=Q$, the left hand side of \eqref{eq:exp-ineq} is $q_i$. For the right hand side, set $L=\range R$. Since every nonzero eigenvalue of $R$ is $-1$, we have $-R=\proj_L$. Since $K$ and $L$ are orthogonal, we have $\proj_K+\proj_L\preceq I_n$. Thus $\E[VV^\T]=\proj_K$ and $V_j\in\sset{0,\pm1}$ give,
    \[
      \Pr(V_j=0)=1-(\proj_K)_{jj}\ge(\proj_L)_{jj}=-R_{jj}=q_j, \text{ for every }j,
    \]
    where the last equality uses that $M$ is hollow. Consequently \eqref{eq:exp-ineq} gives $q\ge\abs{M}q$.

    The vector $q$ is nonnegative because $Q$ is positive semidefinite. Suppose that $q_i=0$, and let $C$ be the connected component containing $i$ in the support graph of $M$. The inequality $q\ge\abs{M}q$ forces $q$ to vanish on $C$. Since $Q$ is positive semidefinite, its principal submatrix on $C$ is zero. Thus the corresponding principal submatrix of $M$ is negative semidefinite, because $R$ is negative semidefinite. It is also hollow, so it must be zero. In particular, the $i$-th row of $M$ is zero.
  \end{claimproof}

  By the claim and \cref{lem:switching}, after switching some of the vectors $v_1,\dots,v_n$, we may assume that $M=-\abs{M}$. Let $C_0$ be the set of isolated vertices in the support graph of $M$, and let $C_1,\dots,C_{m'}$ be its nontrivial connected components.

  The matrix $M$ is block diagonal with respect to $C_0,\dots,C_{m'}$. Let $M_i$ be its diagonal block on $C_i$. The block $M_0$ is zero. For each $i\in\sset{1,\dots,m'}$, the matrix $-M_i$ is nonnegative and irreducible. By the Perron--Frobenius theorem, its largest eigenvalue is positive and simple. Since every positive eigenvalue of $-M$ is $1$, this is the only positive eigenvalue of $-M_i$. Hence each $M_i$ has exactly one negative eigenvalue, equal to $-1$. As $M$ has $m$ negative eigenvalues, we have $m'=m$.

  Each $M_i$ with $i\in\sset{1,\dots,m}$ is hollow and has an eigenvalue of $-1$. Thus \cref{lem:main} gives signs $\eps_i\in\sset{\pm1}^{C_i}$ with $\qf{M_i}{\eps_i}\ge1$. Since $M_0$ is zero and $\qf{M}{\eps}\le m$ for every $\eps\in\qube$, we conclude that $\qf{M_i}{\eps}\le1$ for every $\eps\in\sset{\pm1}^{C_i}$. Applying \cref{lem:rigidity} to each $M_i$ with $i\in\sset{1,\dots,m}$, we find that its order $n_i$ is odd and $M_i=(I_{n_i}-J_{n_i})/(n_i-1)$. Thus the vectors indexed by $C_i$ form a centered regular simplex in an $(n_i-1)$-dimensional subspace. The vectors indexed by $C_0$ are orthonormal, and the subspaces corresponding to distinct blocks are orthogonal. Since $\rank G=d$, these subspaces span $\R^d$.
\end{proof}

\section{Concluding remarks} \label{sec:discussion}

By \cref{thm:equalities}, equality in \cref{thm:main} requires $n-d$ mutually orthogonal simplex blocks, each of dimension at least $2$. Thus $2(n-d) \le d$, or equivalently $n\le 3d/2$, whenever the lower bound in \cref{thm:main} is attained. Since the space of configurations of $n$ unit vectors in $\R^d$ is compact, it follows that
\[
  C(\R^d,n)>\sqrt{2n-d} \qquad\text{for }n>3d/2.
\]
In the plane, Ambrus and Nietert \cite[Propositions 3 and 5]{AN19} determined the exact value:
\[
  C(\R^2,n)=\frac{1}{\sin(\pi/(2n))}, \text{ for every }n.
\]
In particular, $C(\R^2,3)=2$ agrees with \cref{thm:main}, while $C(\R^2,4)=1/\sin(\pi/8)>\sqrt{6}$. Determining $C(\R^d,n)$, and describing its extremizers, for $n>3d/2$ is the main remaining question.

Among dimensions $d\ge3$, the first unresolved case is $(d,n)=(3,5)$, which Fu, Wang, and Yan \cite[Section 2]{FWY23} identified as difficult. For $(d,n)=(3,6)$, the six unit vectors $(1,\pm1,0)/\sqrt2$, $(1,0,\pm1)/\sqrt2$, and $(0,1,\pm1)/\sqrt2$ have largest signed-sum norm $\sqrt{10}$. We conjecture that $C(\R^3,6)=\sqrt{10}$.

Finally, Spencer's bound on balancing unit vectors $c(\R^d,n)\le\sqrt d$ is sharp when $n\ge d$ and $n-d$ is even: add pairs of identical vectors to an orthonormal basis \cite{S81}. In the plane, $c(\R^2,n)=\sqrt2$ for even $n$, while Swanepoel \cite{S00} showed that $c(\R^2,n)=1$ for odd $n$. It would be interesting to determine the optimal constant when $n \ge d \ge 3$ and $n-d$ is odd.

\section*{Acknowledgements}

We thank Alexandr Polyanskii for early discussions with the first author. Research was partially completed while the first two authors were visiting the Institute for Mathematical Sciences, National University of Singapore in August 2026. During the preparation of the current manuscript, it came to our attention that Pinasco \cite{P26} independently resolved the problem for $n = d+1$ using seemingly different techniques. Nevertheless, our result is more general and is heavily inspired by the work of Ball and Prodromou \cite{BP09}.

\bibliographystyle{plain}
\bibliography{vaaler}

\end{document}